\documentclass[a4paper,11pt]{amsart}
\usepackage[T2A]{fontenc}
\usepackage[cp1251]{inputenc}
\usepackage[russian,english]{babel}
\usepackage{amssymb, amsmath, latexsym}
\usepackage{amsfonts,mathtext,cite,enumerate,float,amsthm}
\renewcommand{\le}{\leqslant}
\renewcommand{\ge}{\geqslant}

\newtheorem{theorem}{\indent Theorem}

\newtheorem{lemma}{\indent Lemma}
\theoremstyle{definition}
\newtheorem{example}{\indent Example}
\newtheorem{remark}{\indent Remark}

\renewcommand{\Re}{ {\sf Re}\,}
\renewcommand{\Im}{ {\sf Im}\,}

\usepackage{geometry} 
\begin{document}
\title[Generalized exponential sums with equal weights]{Interpolation by generalized exponential sums\\ with equal weights}
\author{Petr Chunaev}
\address{National Center for Cognitive Technologies, ITMO University (Saint Petersburg, Russia)}
\date{\today}
\begin{abstract}
In this paper we solve Pad\'e- (i.e. multiple) and Prony (i.e. simple exponential) interpolation problems for the {\it generalized exponential sums with equal weights}:
$$
H_n(z; h):=\frac{\mu}{n}\sum\nolimits_{k=1}^n h(\lambda_k z),\quad \text{where}\quad \mu,\lambda_k\in \mathbb{C},
$$
and $h$ is a fixed analytic function under few natural assumptions. The interpolation of a function $f$ by $H_n$ is due to properly chosen $\mu$ and $\{\lambda_k\}_{k=1}^n$, which depend on $f$, $h$ and $n$.

The sums $H_n$ are related to the  {\it $h$-sums} and {\it amplitude and frequency sums} (also known as {\it generalized exponential sums}), i.e. correspondingly to
$$
\mathcal{H}^*_n(z; h):=\sum\nolimits_{k=1}^n \lambda_k h(\lambda_k z)\quad \text{and}\quad
\mathcal{H}_n(z; h):=\sum\nolimits_{k=1}^n \mu_k h(\lambda_k z),\quad \text{where}\quad \mu_k,\lambda_k\in \mathbb{C},
$$
which generalize many classical approximants and whose properties are actively studied.

As for the Pad\'e problem, we show that $H_n$ and $\mathcal{H}_n^*$ have similar constructions and rates of interpolation, whereas calculating $H_n$ requires less arithmetic operations. Although the Pad\'e problem for $\mathcal{H}_n$ is known to have a doubled interpolation rate with respect to $\mathcal{H}_n^*$ and thus to $H_n$, it can be however unsolvable in quite simple and useful cases and this may entirely eliminate the advantage of $\mathcal{H}_n$.  We show that, in contrast to $\mathcal{H}_n$, the Pad\'e problem for $H_n$ always has a unique solution. What is even more important, we also obtain several efficient estimates for $\mu$ and $\lambda_k$, valuable by themselves, and use them in further evaluating interpolation quality and in numerical applications.

The above-mentioned Pad\'e problem and estimates provide a basis for managing the more interesting Prony problem for exponential sums with equal weights $H_n(z;\exp)$, i.e. when $h(z)=\exp(z)$. We show that it is uniquely solvable and surprisingly $\mu$ and $\lambda_k$ can be efficiently estimated. This is in sharp contrast to the case of well-known exponential sums~$\mathcal{H}_n(z;\exp)$.
\end{abstract}

\maketitle
\footnotetext[1]{The research presented in Sections 2, 3 and 5 was funded by Russian Foundation for Basic Research according to the research project 18-01-00744 A. The research presented in Section 4 was financially supported by Russian Science Foundation, Agreement 17-71-30029, with co-financing of Bank Saint Petersburg.}

\section{Introduction}

\subsection{Statement of the problem} In this paper we consider  Pad\'e (i.e. multiple) and Prony (i.e. simple exponential) interpolation by sums of the form
\begin{equation}
\label{equal_AFS}
H_n(z; h):=\frac{\mu}{n}\sum_{k=1}^n h(\lambda_k z),\quad \text{where}\quad \mu,\lambda_k\in \mathbb{C}
\end{equation} 
and $h$ is a fixed analytic function. The interpolation of a function by $H_n$ is carried out by a proper choice of the parameters $\mu$ and $\lambda_k$, $k=1,\ldots,n$, which depend on $n$, $h$ and the function to be interpolated.

The sums (\ref{equal_AFS}) may be considered as representatives of the class of  {\it amplitude and frequency sums} (also known as {\it generalized exponential sums}), i.e. sums with $2n$ free parameters ({\it amplitudes} (or {\it weights}) $\mu_k$ and {\it frequencies} (or {\it exponents}) $\lambda_k$) of the form
\begin{equation}
\label{AFS}
\mathcal{H}_n(z;h):=\sum_{k=1}^n \mu_kh(\lambda_k z),\quad \text{where}\quad \mu_k,\lambda_k\in \mathbb{C}.
\end{equation}
The approximative properties of general sums (\ref{AFS}) and their particular cases (including exponential sums, classical Pad\'e approximants, Gauss type quadratures) are actively studied in approximation theory (see a brief survey e.g. in \cite{DanChu2016}). The sums (\ref{AFS}) with the restriction $\mu_k=\lambda_k^\eta$ for some $\eta\in \mathbb{N}_0$ (i.e. already with $n$ free parameters)
\begin{equation}
\label{h-sums-nu}
\mathcal{H}^*_{\eta,n}(z;h):=\sum_{k=1}^n\lambda_k^\eta h(\lambda_k z),\qquad \eta\in \mathbb{N}_0,\qquad \lambda_k\in\mathbb{C},
\end{equation}
are usually called {\it $h$-sums}; they were introduced in \cite{Dan2008}. The most explored case is ${\eta=1}$,
\begin{equation}
\label{h-sums}
\mathcal{H}^*_n(z;h):=\mathcal{H}^*_{1,n}(z;h)=\sum_{k=1}^n\lambda_k h(\lambda_k z),\qquad \lambda_k\in\mathbb{C},
\end{equation}
 see \cite{Dan2008,Chu2010,DanChu2011,Fryantsev,Borodin,Nigmatyanova,Chu_extrapolation}. The paper \cite{Dan2008} contains several remarks\footnote{Note that the sums (\ref{h-sums}) for $\eta=0$, although look similar, have more restricted approximative properties  than our sums (\ref{equal_AFS}). Indeed, the $0$th Taylor coefficient of the function 
 	$\mathcal{H}^*_{0,n}(z;h)$ is always $n$. This does not allow to approximate functions $f$ with $f_0\neq n$. This circumstance can be however overcome by considering $\frac{n}{f_0}f$ with $f_0\neq 0$ instead of $f$, i.e. by applying our sums (\ref{equal_AFS}) in fact.
 	The sums (\ref{equal_AFS}) and (\ref{h-sums}) are also connected as follows. Put $h(z)=z^\eta g(z)$, $\eta\in \mathbb{N}_0$, in (\ref{equal_AFS}) to get $H_n(z;h)=\frac{\mu}{n}\sum_{k=1}^n (\lambda_k z)^\eta g(\lambda_k z) =z^\eta  \frac{\mu}{n}\mathcal{H}^*_{\eta,n}(z;g)$.} on the general case of~(\ref{h-sums-nu}).

Observe that computing (\ref{equal_AFS}) requires less arithmetic operations than that of (\ref{AFS}) and (\ref{h-sums}), although all these sums have similar approximative properties with respect to the number of free parameters, as will be shown below.

Let us come back to the formulation of the problems that we consider in this paper. In the case of the Pad\'e interpolation we set
\begin{equation}
\label{fh}
f(z)=\sum_{m=0}^\infty f_m z^m\quad \text{and} \quad h(z)=\sum_{m=0}^\infty h_m z^m.
\end{equation}
The function $h$ is a fixed analytic function, $f$ is an analytic function to be interpolated. Additionally, we suppose\footnote{If the first non-vanishing Taylor coefficient of $f$ is $f_l$, then write $f(z)=z^lF(z)$ so that $F_0\neq 0$ and apply the scheme from this paper to $F$ to get an interpolant of the form $z^l\frac{\mu}{n}\sum_{k=1}^n h(\lambda_k z)$ for $f(z)$. In some cases it is reasonable to add a non-zero parameter playing the role of $f_0\neq 0$ (see Subsection~\ref{diff_section}).} that
\begin{equation}
\label{fh-condition}
f_0\neq 0\qquad \text{and}\qquad h_m=0 \Rightarrow f_m=0\quad \text{for all }m=0,1,\ldots.
\end{equation}

For convenience, we introduce the following (well defined due to (\ref{fh-condition})) numbers:
\begin{equation}
\label{r_m}
r_{m}=r_m(f,h):=\left\{
\begin{array}{ll}
0, & f_{m}=0,\\
f_m/h_m, & f_{m}\ne 0,\\
\end{array}
\right.
\qquad m=0,1,\ldots.
\end{equation}

We are interested in solving the following {\bf Pad\'e (multiple) interpolation problem in a neighbourhood of $z=0$}: {\it find complex $\mu$ and $\{\lambda_k\}_{k=1}^n$, depending on $f,h$ and $n$, such that}
\begin{equation}
\label{Pade-problem}
f(z)-H_n(z;h)=O(z^{n+1})\quad\text{for}\quad z\to 0.
\end{equation}

As for the Prony interpolation, we fix $h(z)=\exp(z)$ in (\ref{equal_AFS}) and interpolate by 
\begin{equation}
\label{exp-sums}
H_n(z;\exp)=H_n^{\sf exp}(z):=\frac{\mu}{n}\sum_{k=1}^n \exp(\lambda_k z)
\end{equation}
the table
\begin{equation}
\label{table}
\left\{m, g(m)\right\}_{m=0}^n,\qquad g(0)\neq 0,
\end{equation}
generated by a complex-valued function $g$. Thus we deal with the {\bf Prony (simple exponential) interpolation problem}: {\it find complex $\mu$ and $\lambda_k$, depending on $g$ and $n$, such that}
\begin{equation}
\label{Prony-problem}
g(z)=H_n^{\sf exp}(z)\quad \text{for}\quad z\in\{m\}_{m=0}^{n}.
\end{equation}

\medskip

The paper is organised as follows. Section~\ref{section2} (with an appendix in Section~\ref{section_proof_of_theorem1}) contains several estimates for the so-called power sums and their components. The estimates have their own value and  are used later for estimating $|\mu|$, $|\lambda_k|$, the remainder and the rate of interpolation in the problems under consideration. Section~\ref{section3.1} is devoted to solving the Pad\'e problem (\ref{Pade-problem}), with corresponding estimates. In Sections~\ref{section3.2} and \ref{section3.3}, we compare approximative properties of $H_n$ with those of $\mathcal{H}_n$ and $\mathcal{H}^*_n$ and give several applications of $H_n$ to numerical analysis. In Section~\ref{section4.1} we solve the  Prony  problem (\ref{Prony-problem}) and estimate the interpolation parameters. In Sections~\ref{section4.2} and~\ref{section4.3}, we compare $H_n^{\sf exp}$, solving (\ref{Prony-problem}), and the original Prony exponential sums.

\section{Estimates for power sums and their components}
\label{section2}

We first aim to prove several estimates for the power sums of complex numbers. They are of an independent interest since are related to the power sums problems appearing in different fields of analysis (e.g. in Tur\'an's power sum method).
Let 
\begin{equation}
\label{Lambda}
\Lambda_n:=\{\lambda_k\}_{k=1}^n,\qquad \text{where} \quad \lambda_k\in\overline{\mathbb{C}}.
\end{equation}
Consider the {\it power sums} for the set $\Lambda_n$:
\begin{equation}
S_{m}:=S_{m}(\Lambda_n)=\sum^n_{k=1}\lambda^{m}_k, \qquad m=1,2,\ldots. 
\label{S_m} 
\end{equation}
\begin{theorem}
\label{theorem1}
Let $n\ge 2$.
If $|S_m(\Lambda_n)|\le a^m$ for some $a\ge 0$ and $m=1,\ldots,n$, then
\begin{equation}
\label{main_lambda_est}
\max_{k=1,\ldots,n}|\lambda_k|\le (1+\varepsilon_n)a,\qquad \text{where}\quad \varepsilon_n:=\frac{2(\ln n -\ln\ln n)}{n}< \frac{2\ln n}{n}.
\end{equation}
Furthermore, $(\ref{main_lambda_est})$ cannot be improved much as for $n\ge n_0$ there exists $\tilde{\Lambda}_n$ such that
\begin{equation}
\label{varepsilon_improve}
|S_m(\tilde{\Lambda}_n)|\le a^m,\quad m=1,\ldots,n,\quad\text{and}\quad |\tilde{\lambda}_1|= \left(1+\frac{c_n}{n}\right)a,\quad c_n\in[1/10,1].
\end{equation}
\end{theorem}

This result is a revised and generalized version of the estimates partly obtained in the papers \cite{Chu2010,DanChu2011} and the unpublished manuscript \cite{Chu2013} by the author. The preceding and more qualitative estimate $\max_{k=1,\ldots,n}|\lambda_k|\le 2a$ under the same assumptions is proved in \cite{Dan2008}. The proof of Theorem~\ref{theorem1} is postponed to Section~\ref{section_proof_of_theorem1} due to its length.

Below we will use Theorem~\ref{theorem1} to obtain estimates for different parameters in the interpolation processes under consideration. 

For further discussion we recall how to find the set $\Lambda_n$ (see (\ref{Lambda})) from the following  system for their power sums $S_m=S_m(\Lambda_n)$:
\begin{equation}
\label{Newton_MP}
S_m = s_m,\quad m=1,\ldots,n,\qquad \text{where }s_m\in \mathbb{C} \text{ are given}. 
\end{equation} 
We call (\ref{Newton_MP}) a {\it Newton moment problem}. To proceed, let us introduce the {\it elementary symmetric polynomials} for the elements of $\Lambda_n$:
\begin{equation}
\label{sym_polynomials}
\sigma_m=\sigma_m(\Lambda_n):=\sum_{1\le j_1 < \ldots< j_m \le
	n}\lambda_{j_1}\cdots \lambda_{j_m},\qquad m=1,\ldots,n.
\end{equation}
The connection between the power sums (\ref{S_m}) and polynomials (\ref{sym_polynomials}) is expressed by the well-known \textit{Newton-Girard formulas} \cite[Section 3.1]{Prasolov}:
\begin{equation}
\label{sigma}
\sigma_1=S_1,\quad
\sigma_m=\frac{(-1)^{m+1}}{m}
\left(S_{m}+\sum_{j=1}^{m-1}
(-1)^{j}\,S_{m-j}\sigma_j\right),\quad m=2,\ldots,n.
\end{equation}
Moreover, the set $\Lambda_n$ is formed by the $n$ roots of the {\it unitary} polynomial
\begin{equation}
\label{P_n}
P_n(\lambda):=\lambda^n-\sigma_1 \lambda^{n-1}+\sigma_2
\lambda^{n-2}+\ldots+(-1)^n\sigma_n.
\end{equation}

Consequently, given any $s_m$, one can solve the system (\ref{Newton_MP}) using (\ref{sigma}) and (\ref{P_n}) and --- what is very important for us --- this {\it solution $\Lambda_n$ always exists and is unique}.

The formulas (\ref{sigma}) and (\ref{P_n}) allow to get estimates for $|\sigma_m|$ and $|\lambda_k|$ (as in Theorem~\ref{theorem1}) under some assumptions on $|s_m|$ (i.e. on $|S_m|$, equivalently). For example, it is proved in \cite{Chu2010}, that the condition $|s_m|\le a^m$ implies that $|\sigma_m|\le a^m$, where $m=1,\ldots,n$. This is applied in \cite{Dan2008,Chu2010,Chu2013} for obtaining several estimates preceding to (\ref{main_lambda_est}). As shown in \cite{DanChu2011}, the condition $s_m= a^{m-1}$ for $a\ge 1$ implies that $\max_{k=1,\ldots,n}|\lambda_k|\le a(1-(1-a^{-1})/n)$. This estimate is essentially used for constructing new extrapolation formulas for analytic functions in \cite{Chu_extrapolation,DanChu2011}. Other related estimates can be also found e.g. in \cite{Komarov2018,Fryantsev,Dan2008}. Note that the majority of previous estimates are established under the condition that power sums are bounded by corresponding members of a {\it geometric} progression. Now we prove a result with another condition that in particular gives the case when the power sums are bounded by members of an {\it arithmetic} progression. 
\begin{theorem}
\label{theorem_new_estimate}
Let $n\ge 2$. 
If $|S_m|\le \gamma m^va^m$ with some $\gamma>0$, $a\ge 0$ and $v\in[0,1]$ for all $m=1,\ldots,n$, then
\begin{equation}
\label{sigma_est_new}
|\sigma_m|\le \gamma m^{v-1}(1+\gamma)^{m-1}a^m,\qquad m=1,\ldots,n.
\end{equation}
Moreover, it holds that
\begin{equation}
\label{sigma_lambda_est_new}
\max_{k=1,\ldots,n}|\lambda_k|\le \left((1+\gamma)n^{\frac{v-1}{n-1}}+\gamma\right) a\le \left(1+2\gamma\right) a.
\end{equation}
\end{theorem}
\begin{proof}
First, by the change of variables the problem can be reduced to the case $a=1$. We proceed by induction. For $m=1$ we get from (\ref{sigma}) that $|\sigma_1|\le \gamma$ and thus (\ref{sigma_est_new}) holds in this case. Suppose that (\ref{sigma_est_new}) is also true for each $m=2,\ldots,M-1$. Then by (\ref{sigma}),
\begin{align*}
M&|\sigma_M|\\
&\le |S_{M}|+\sum\nolimits_{j=1}^{M-1}
|S_{M-j}||\sigma_j|\le \gamma M^v+\gamma^2\sum\nolimits_{j=1}^{M-1}
 (M-j)^v j^{v-1}(1+\gamma)^{j-1}\\
& \le \gamma M^{v}
\left(1+\gamma\sum\nolimits_{j=1}^{M-1}
\left(1-\tfrac{j}{M}\right)^v j^{v-1}(1+\gamma)^{j-1}\right)\le
\gamma M^{v}
\left(1+\gamma\sum\nolimits_{j=1}^{M-1}
(1+\gamma)^{j-1}\right)\\
&=\gamma M^{v}
\left(1+\gamma\cdot \frac{1-(1+\gamma)^{M-1}}{1-(1+\gamma)}\right)
=\gamma M^{v}(1+\gamma)^{M-1}.
\end{align*}
Dividing both parts by $M$ yields the required inequality for $|\sigma_M|$.

Now we prove the estimate for $|\lambda_k|$. From (\ref{P_n}) and (\ref{sigma_est_new}) we get for $\lambda\neq 0$ that
$$
\frac{|P_n(\lambda)|}{|\lambda|^n}\ge 1-\sum\nolimits_{m=1}^n\frac{|\sigma_m|}{|\lambda|^m}\ge 1-\frac{\gamma}{|\lambda|}\sum\nolimits_{m=1}^n\left(\frac{1+\gamma}{|\lambda|}\right)^{m-1}m^{v-1}.
$$
Furthermore,
$m=(e^{\frac{\ln m}{m-1}})^{m-1}\ge (e^{\frac{\ln n}{n-1}})^{m-1}=n^{\frac{m-1}{n-1}}$ for $m=2,\ldots,n$. Since $v\in[0,1]$, it holds for $|\lambda|>(1+\gamma)n^{\frac{v-1}{n-1}}+\gamma$ that
$$
\frac{|P_n(\lambda)|}{|\lambda|^n}\ge
1-\frac{\gamma}{|\lambda|}\sum\nolimits_{m=1}^n\left(\frac{1+\gamma}{|\lambda|}n^{\frac{v-1}{n-1}}\right)^{m-1}=
$$
$$ 
=\frac{|\lambda|-(1+\gamma)n^{\frac{v-1}{n-1}}-\gamma\left(1-\left(\frac{1+\gamma}{|\lambda|}n^{\frac{v-1}{n-1}}\right)^n\right)}{|\lambda|-(1+\gamma)n^{\frac{v-1}{n-1}}}
\ge 
\frac{|\lambda|-\left((1+\gamma)n^{\frac{v-1}{n-1}}+\gamma\right)}{|\lambda|-(1+\gamma)n^{\frac{v-1}{n-1}}}>0.
$$
Thus, all $\lambda_k$, the roots of $P_n$, lie in the disc $|\lambda|\le (1+\gamma)n^{\frac{v-1}{n-1}}+\gamma$. For the second inequality in (\ref{sigma_lambda_est_new}), take into account that $v\in[0,1]$ and $n^{\frac{v-1}{n-1}}\le 1$ for $n\ge 2$.
\end{proof}

\section{Pad\'e interpolation by $H_n$}

The results from this section were announced in \cite{Chu_Conf}. 

\subsection{Main theorem about the Pad\'e interpolation by $H_n$}
\label{section3.1}

Recall the definitions (\ref{equal_AFS}) and (\ref{Pade-problem}) and the assumptions (\ref{fh}), (\ref{fh-condition}) and (\ref{r_m}).

\begin{theorem}
\label{theorem2}
Fix $n$ and $h$. Given a function $f$ satisfying $(\ref{fh})$ and $(\ref{fh-condition})$, there exist uniquely determined $\mu\neq 0$ and $\Lambda_n=\{\lambda_k\}_{k=1}^n$ such that the following interpolation formula holds:
\begin{equation}
\label{th1_equality} f(z)=\frac{\mu}{n}\sum_{k=1}^nh(\lambda_k	z)+R_n(z),\qquad R_n(z)=O(z^{n+1}),\qquad z\to 0.
\end{equation}
This formula is exact for polynomials $f$ of degree $\le n$, i.e. $R_n(z)\equiv 0$ for such~$f$.

More precisely, one can find the above-mentioned numbers as follows:
\begin{equation}
\label{mu-Pade}
\mu=r_0;
\end{equation}
$\Lambda_n$ is the solution to the system of the form $(\ref{Newton_MP})$ with
\begin{equation}
\label{syst2}
s_{m}=\frac{n}{r_0}\,r_m, \qquad m=1,\ldots,n.
\end{equation}
\end{theorem}
\begin{proof}
By (\ref{equal_AFS}) and (\ref{fh})
$$
H_n(z;h)=\frac{\mu}{n}\sum_{k=1}^n\sum_{m=0}^\infty h_m(\lambda_k
z)^m=\sum_{m=0}^\infty
h_m\left(\frac{\mu}{n}\sum_{k=1}^n\lambda_k^m\right)z^m.
$$
From the condition (\ref{Pade-problem}), i.e.
$$
\sum_{m=0}^\infty
h_m\left(\frac{\mu}{n}\sum_{k=1}^n\lambda_k^m\right)z^m=\sum_{m=0}^\infty f_mz^m+O(z^{n+1}),
$$ 
we arrive at the system
$$
h_m\left(\frac{\mu}{n}\sum_{k=1}^n\lambda_k^m\right)=f_m,\qquad m=0,\ldots,n.
$$
From here, by taking into account (\ref{fh-condition}), we obtain the unique  $\mu\neq 0$ as in (\ref{mu-Pade}) and the system (\ref{syst2}) that is actually a Newton-type moment problem (\ref{Newton_MP}), whose solution $\Lambda_n=\Lambda_n(\{s_m\}_{m=1}^n)$ always exists and is unique.
\end{proof}

For the terms in the next result, recall Theorems~\ref{theorem1} and~\ref{theorem2}.

\begin{theorem}
\label{theorem_Pade_remainder}
Suppose that the assumptions of Theorem~\ref{theorem2} are satisfied. Additionally, let $|h_m|\le 1$ for all $m=n+1,n+2,\ldots$. Then the following holds for $(\ref{th1_equality})$:

\begin{itemize}
	\item[$(a)$] If $|r_m|\le \frac{|r_0|}{n}\,a^m$ for all $m=1,2,\ldots$ and some $a\ge 0$, then 
	\begin{itemize}
		\item[$(i)$] $\max_{k=1,\ldots,n}|\lambda_k|\le (1+\varepsilon_n)a$,
		\item[$(ii)$] in the disk $|z|<(1+\varepsilon_n)^{-1}a^{-1}$ the sum $H_n(z)$ is analytic  and moreover
		\begin{equation}
		\label{R_n_estimated}
		|R_n(z)|\le \frac{2|r_0| n^2|az|^{n+1}}{1-(1+\varepsilon_n)a|z|},\qquad n\ge 2,
		\end{equation}
		\item[$(iii)$] $H_n(z)\to f(z)$ uniformly for $|z|< a^{-1}$.
	\end{itemize}
	\item[$(b)$] If $|r_m|\le \frac{|r_0|}{n}\,\gamma m a^m$ for all $m=1,2,\ldots$ and some $\gamma>0$ and $a\ge 0$, then 
	\begin{itemize}
		\item[$(i)$] $\max_{k=1,\ldots,n}|\lambda_k|\le (1+2\gamma)a$,
		\item[$(ii)$] in the disk $|z|<(1+2\gamma)^{-1}a^{-1}$ the sum $H_n(z)$ is analytic and moreover
		\begin{equation}
		\label{R_n_estimated_2}
		|R_n(z)|\le \frac{2|r_0||(1+2\gamma)az|^{n+1}}{(1-(1+2\gamma)a|z|)^2},\qquad n\ge 2,
		\end{equation}
		\item[$(iii)$] $H_n(z)\to f(z)$ uniformly for $|z|<(1+2\gamma)^{-1}$.
	\end{itemize}
\end{itemize}
\end{theorem}
\begin{proof}
By the change of variables we can reduce the proof to the case of $a=1$. 

Let us start with $(a)$. Since $|r_m|\le \frac{|r_0|}{n}$ for $m=1,2,\ldots$, we have $|s_m|\le 1$ for $m=1,\ldots,n$ and therefore $|\lambda_k|\le 1+\varepsilon_n$ by Theorem~\ref{theorem1}. This implies that $|S_m|\le n(1+\varepsilon_n)^m$ for $m\ge n+1$. Recall the definition (\ref{r_m}) and that $\mu=r_0$.
Consequently, taking into account all the assumptions,
$$
|R_n(z)|= \left|\sum_{m=n+1}^\infty \left(r_m -\frac{r_0}{n}S_m\right)h_mz^m\right|\le\frac{|r_0|}{n} \sum_{m=n+1}^\infty \left(1 +n(1+\varepsilon_n)^m\right)|z|^m
$$
$$
\le \frac{|r_0|}{n} \left(\frac{|z|^{n+1}}{1-|z|}+ \frac{n|(1+\varepsilon_n)z|^{n+1}}{1-(1+\varepsilon_n)|z|}\right)
\le \frac{|r_0|(1/n+(1+\varepsilon_n)^{n+1})|z|^{n+1}}{1-(1+\varepsilon_n)|z|},\quad |z|<\frac{1}{1+\varepsilon_n}.
$$
To get (\ref{R_n_estimated}),  note that $1/n+(1+\varepsilon_n)^{n+1}\le 2 n^2$ for $n\ge 2$.

For $|z|\le (1-\delta)(1+\varepsilon_n)^{-1}$, where $\delta\in(0,1)$, we get
$$
|R_n(z)|\le 2|r_0|n^2 (1-\delta)^{n+1}/\delta.
$$
This implies that $|R_n(z)|\to 0$ uniformly for $|z|<1$, recalling that $\varepsilon_n\to 0$ as $n\to \infty$.

\medskip

Now we consider $(b)$. Since $|r_m|\le \frac{|r_0|}{n}\gamma m$ for $m=1,2,\ldots$, we have $|s_m|\le \gamma m$ for $m=1,\ldots,n$ and therefore $|\lambda_k|\le 1+2\gamma$ by Theorem~\ref{theorem_new_estimate}. This implies that $|S_m|\le n(1+2\gamma)^m$ for $m\ge n+1$.
Consequently, if $|z|<(1+2\gamma)^{-1}$, then
$$
|R_n(z)|\le\frac{|r_0|}{n} \sum_{m=n+1}^\infty \left(\gamma m +n(1+2\gamma)^m\right)|z|^m
$$
$$
=\frac{|r_0|}{n} \left(\gamma\frac{(n+1-n|z|)|z|^{n+1}}{(1-|z|)^2} +n\frac{|(1+2\gamma)z|^{n+1}}{1-(1+2\gamma)|z|}\right)
\le \frac{2|r_0||(1+2\gamma)z|^{n+1}}{(1-(1+2\gamma)|z|)^2}. 
$$

For $|z|\le (1-\delta)(1+2\gamma)^{-1}$, where $\delta\in(0,1)$, we get
$$
|R_n(z)|\le 2|r_0| (1-\delta)^{n+1}/\delta^2.
$$
This implies that $|R_n(z)|\to 0$ uniformly for $|z|<(1+2\gamma)^{-1}$.
\end{proof}

\subsection{The number of arithmetic operations. Comparison with other Pad\'e-type problems for amplitude and frequency sums}
\label{section3.2}
From the point of view of necessary arithmetic operations, calculating the amplitude and frequency sums (\ref{AFS}) and $h$-sums (\ref{h-sums}) for each fixed $z$ and known $\lambda_k$, $\mu_k$ and $h(\lambda_kz)$ requires,
generally speaking, $n$ multiplications ($\mu_k$ or
$\lambda_k$ by $h(\lambda_k z)$) and $n$ summations (the sum of the values obtained). On the other hand, calculating the sums (\ref{equal_AFS}) requires $n$ summations and just one multiplication (additionally note that $\mu$ is independent of $n$ and is only determined by $r_0$). This reduction in arithmetic complexity lies in the circle of problems considered by P. Chebyshev. In particular, this was his motivation in obtaining the famous quadrature with equal weights, see \cite[Section 10, \S3]{Krylov} and \cite[Section VI, \S4]{Natanson}. We will come back to this quadrature in Section~\ref{Chebyshev_quadrature} in the context of the sums (\ref{equal_AFS}).

\medskip
Now we compare Theorems~\ref{theorem2} and~\ref{theorem_Pade_remainder} with the corresponding ones for (\ref{AFS}) and (\ref{h-sums}). The result for (\ref{h-sums}) is proved in \cite{Dan2008} and can be summarised as follows under assumptions of Theorem~\ref{theorem2}: {\it there always  exists a uniquely determined set $\Lambda_n=\{\lambda_k\}_{k=1}^n$ such that}
 $$
 f(z)=\sum_{k=1}^n\lambda_kh(\lambda_k z)+R_n(z),\qquad R_n(z)=O(z^n),\qquad z\to 0.
 $$
 The set $\Lambda_n$ is the solution to the system (\ref{Newton_MP}) with
\begin{equation}
\label{syst_h_sums}
s_{m}=r_{m-1}, \qquad m=1,\ldots,n.
\end{equation}

Thus, it can be seen that the Pad\'e interpolation schemes for $H_n$ and $h$-sums $\mathcal{H}^*_n$ are similar. In both cases the solution always exists and is unique under the assumptions of Theorem~\ref{theorem2}. Moreover, the corresponding rates of interpolation ($O(z^n)$ and $O(z^{n+1})$) just slightly differ and directly depend on the number of free parameters. This similarity clearly underlines the advantage of $H_n$ over $\mathcal{H}^*_n$ in the sense of the number of required arithmetic operations discussed at the beginning of this subsection.

\medskip

The Pad\'e interpolation problem for the amplitude and frequency sums (\ref{AFS}) is more delicate. First of all, it is {\it not always solvable} for given $f$ and fixed $h$ and $n$, even if the assumptions of Theorem~\ref{theorem2} are met. As shown in \cite{DanChu2016}, its solvability relies on the properties of the (possibly {\it non-unitary}) polynomial
\begin{equation}
\label{generating_polynomial}
P^*_n(\lambda):=\sum_{m=0}^n
\sigma^*_{m} \lambda^{m}= \left|
\begin{array}{ccccc}
1 & \lambda & \lambda^2 & \ldots & \lambda^n\\
s_0 & s_1 & s_2 & \ldots & s_n\\
s_1 & s_2 & s_3 & \ldots & s_{n+1}\\
\ldots & \ldots & \ldots & \ldots & \ldots\\
s_{n-1} & s_n & s_{n+1} & \ldots & s_{2n-1}\\
\end{array}
\right|,
\end{equation}
which is an analogue of (\ref{sigma}) and (\ref{P_n}) for the following weighted version of (\ref{Newton_MP}),
\begin{equation}
\label{Prony_MP}
S^*_m = s_m,\qquad \text{where}\qquad S^*_m:=\sum_{k=1}^n\mu_k\lambda_k^m,\qquad m=0,\ldots,2n-1,
\end{equation} 
and the moments $s_m$ are defined as follows:
\begin{equation}
\label{syst_AFSs}
s_{m}=r_{m},\qquad m=0,\ldots,2n-1.
\end{equation}
Namely, it is proved in \cite{DanChu2016} that for the functions $f$ and $h$ satisfying the assumptions of Theorem~\ref{theorem2}, it holds with {\it uniquely determined $\{\mu_k,\lambda_k\}_{k=1}^n$} that
\begin{equation}
\label{Pade_AFS}
f(z)=\sum_{k=1}^n\mu_k h(\lambda_k z)+R_n(z),\qquad R_n(z)=O(z^{2n}),
\end{equation}
if and only if {\it the polynomial $(\ref{generating_polynomial})$ is of degree $n$ and all its roots are pairwise distinct}. 
This condition on $P_n^*$ is quite strong and can be unsatisfied even for simple and natural sequences of moments $s_m$ in (\ref{syst_AFSs}), e.g. $s_m=m+1$ or $s_m=a^m$, where $m=0,1,\ldots$ (see \cite{DanChu2016}), whilst the corresponding problems for (\ref{equal_AFS}) and (\ref{h-sums}) still have unique solutions. This disadvantage of the amplitude and frequency sums (\ref{AFS}) with respect to $H_n$ and $h$-sums is however quite compensated by the doubled rate of interpolation, $O(z^{2n})$. 

Note that the system (\ref{Prony_MP}) and particular cases of the identity (\ref{Pade_AFS}) appear in different areas of analysis and approximation theory and are closely related to Hankel matrices, Gauss quadratures, classical Pad\'e fractions, exponential sums and Hamburger, Stieltjes and Hausdorff moment problems. A survey on these connections can be found e.g. in \cite[Section 2]{DanChu2016} or \cite{Lyubich}. Moreover, the system (\ref{Prony-problem}) is the main tool to solve the {\it original Prony $($simple exponential$)$ interpolation problem}
\begin{equation}
\label{Prony_exponential}
\mathcal{H}_n(m;\exp)=\sum_{k=1}^n\mu_k \exp(\lambda_k m)=g(m),\quad \mu_k\in \mathbb{C}\setminus \{0\},\,\,\lambda_k\in \overline{\mathbb{C}},\quad m=0,\ldots,2n-1,
\end{equation}
where $\lambda_k$ are assumed pairwise distinct. Let us mention that (\ref{Prony_exponential})
 is well-studied analytically and has numerous applications (see \cite{Lyubich,Prony_method_survey_2,DanChu2016} for a nice survey). Moreover, there are several numerical approaches for solving (\ref{Prony-problem}) and its variations, see  \cite{Numerical_Prony_method_survey_1,Numerical_Prony_method_survey_2,Numerical_Prony_method_survey_3,Numerical_Prony_method_survey_4}. But still, from the above-mentioned condition on (\ref{generating_polynomial}) one can deduce that {\it the Prony problem $(\ref{Prony_exponential})$ can have no solution for some $g(m)$, $m=0,\ldots,2n-1$}. We will come back to this issue in Section~\ref{Prony_section}.

\subsection{Applications of the Pad\'e interpolation and corresponding estimates}
\label{section3.3}
Now we give several examples how Theorem~\ref{theorem2} can be applied in numerical analysis. We compare these applications with the corresponding ones for (\ref{AFS}) and (\ref{h-sums}) in appropriate places.

\subsubsection{The case $f(z)=h(az)$ for a complex $a\neq 0$} Under assumptions of Theorem~\ref{theorem2}, we  have $\mu=1$ and $s_{m}=na^m$, $m=1,\ldots,n$.
The solution to (\ref{Newton_MP}) is then
$\Lambda_n=\{a\}_{k=1}^n$. Thus
$$
f(z)=h(az)\equiv H_n(z)=\frac{1}{n}\sum_{k=1}^nh(az),\qquad \text{i.e. } R_n(z)\equiv 0.
$$
In particular, this means that $H_n$ do not generate  extrapolation operators appearing in a~similar situation for $h$-sums as in \cite{DanChu2011,Chu_extrapolation}.

\subsubsection{Rational interpolation} Choosing $h(z)=1/(z-1)=-\sum_{m=0}^\infty z^m$ in Theorem~\ref{theorem2} leads to rational interpolants of the form
$$
H_n^{\sf rat}(z):=\frac{\mu}{n} \sum_{k=1}^{n}\frac{1}{\lambda_kz-1}.
$$
For example, if $f(z)\equiv 1$, then
$\mu=-1$, $s_m=0$ for $m=1,\ldots,n$ and thus $\Lambda_n=\{0\}_{k=1}^n$. Consequently, the corresponding 
$H_n^{\sf rat}(z)=-\frac{1}{n} \sum_{k=1}^{n}\frac{1}{0\cdot z-1}\equiv 1$,
i.e. our interpolant coincides with $f$. Such a coincidence clearly happens for all $f(z)=H_n^{\sf rat}(z)$ with arbitrarily chosen $\{\lambda_k\}_{k=1}^n$ due to the uniqueness of $\Lambda_n$.

If $h(z)=1/(z-1)$ in (\ref{AFS}) and (\ref{h-sums}), then interpolants to~$f$ are correspondingly the well-known $[n-1,n]$--type Pad\'e fractions (see e.g. \cite[Subsection 2.3]{DanChu2016}) and rational $h$-sums called {\it simple partial fractions} whose properties are actively studied \cite{survey}. In comparison with $H_n^{\sf rat}$, the calculation of the $h$-sums  require more arithmetic operations  whilst the Pad\'e fractions may not exist for some $f$ and $n$ (see Section~\ref{section3.2}).

\subsubsection{Pad\'e interpolation by exponential sums} Another important particular case of $H_n$ is when one chooses $h(z)=\exp(z)$ and obtains Pad\'{e} exponential sums of the form (\ref{exp-sums}), i.e.
$$
H_n^{\sf exp}(z)=\frac{\mu}{n} \sum_{k=1}^{n}\exp(\lambda_kz).
$$

Let us interpolate $f(z)=\cos(z)$ by $H_2^{\sf exp}(z)$. We have 
$$
\mu=1,\qquad s_1=0,\qquad s_2=-2,\qquad P_2(\lambda)=\lambda^2+1.
$$
Consequently, $\Lambda_2=\{i,-i\}$ and we get the well-known identity
$$
\cos z=H_2^{\sf exp}(z)=\frac{\exp(iz)+\exp(-iz)}{2}.
$$
Surprisingly, this identity appears for  $f$ and $h$ mentioned for any even $n$ in $H_n^{\sf exp}(z)$. 

\medskip

Let us emphasize that interpolants $H_n^{\sf exp}$ always exist for a given $f$ with ${f_0\neq 0}$ (since the condition (\ref{fh-condition}) is always satisfied), unlike exponential sums $\mathcal{H}_n(z;\exp)$ of the form (\ref{AFS}) (see Section~\ref{section3.2}).

\subsubsection{Chebyshev's quadrature}
\label{Chebyshev_quadrature} Let us use $H_n$ to interpolate the function
\begin{equation}
\label{integral}
f(x)=\frac{1}{x}\int_{-x}^x h(t)\rho(t)dt, \qquad x>0,
\end{equation}
where $f$ and $h$ satisfy (\ref{fh}) and the {\it integral weight} $\rho=\rho(t)\ge 0$ for $t\in [-x,x]$.

As an example, take $\rho(x)\equiv 1$. Then by (\ref{fh}) clearly
$$
f(x)=\frac{1}{x}\int_{-x}^x h(t)\, dt=
\sum_{m=0}^\infty h_m \left(\frac{1}{x}\int_{-x}^x t^m dt\right)=\sum_{m=0}^\infty \frac{1+(-1)^m}{m+1}h_m z^m,
$$
and from Theorem~\ref{theorem2} we deduce that $\mu =2$ and $\Lambda_n$ is the solution to the system (\ref{Newton_MP}) with 
\begin{equation}
\label{Cheb_syst}
s_m=\frac{n}{2}\cdot\frac{1+(-1)^m}{m+1}, \qquad m=1,\ldots,n.
\end{equation}
Note that $\mu$ and $\Lambda_n$ are independent of $h$ and are universal in this sense. The system (\ref{Newton_MP}) with (\ref{Cheb_syst}) and the corresponding polynomials $P_n$ of the form (\ref{P_n}) are well studied \cite[Section 10, \S3]{Krylov}. Thus for a fixed $x>0$ one gets the interpolation formula
\begin{equation}
\label{Cheb_quadr}
\frac{1}{x}\int_{-x}^x h(t)dt=\frac{2}{n}\sum_{k=1}^nh(\lambda_k x) +R_n(x),
\end{equation}
that is nothing else but \textit{Chebyshev's quadrature with equal weights} \cite[Section 10, \S3]{Krylov},
whose frequencies $\lambda_k$, the roots of $P_n$, are real and belong to the segment $[-1,1]$ only for $n=1,\ldots,7,9$. For other $n$ there are complex $\lambda_k$ in (\ref{Cheb_quadr}). In particular, this is proved by S. Bernstein for $n\ge 10$. Further information on the distribution of $\lambda_k$ can be found in \cite{Kuzmin,Natanson}. What is more, Theorem~\ref{theorem2} implies that (\ref{Cheb_quadr}) is exact for polynomials $h$ of degree $\le n$ in the sense that $R_n(x)\equiv 0$ for such $h$. Moreover, for even $n$ the quadrature formula is exact for polynomials $h$ of degree $\le n+1$ as $S_{n+1}=0$. One can find more information on (\ref{Cheb_quadr}), including estimates for the remainder $R_n(x)$, in \cite[Section 10, \S3]{Krylov}.

\medskip

Let us briefly mention that if we use $H_n$ to interpolate the function
$$
f(x)=\frac{1}{x}\int_{0}^x h(t)dt, \qquad x>0,
$$
then $\mu =1$ and $\Lambda_n$ is the solution to the system (\ref{Newton_MP}) with 
\begin{equation}
\label{Cheb_syst_01}
s_m=\frac{n}{m+1}, \qquad m=1,\ldots,n.
\end{equation}
Thus, for a fixed $x>0$, one gets {\it shifted Chebyshev's quadrature}
\begin{equation}
\label{Cheb_quadr_shifted}
\frac{1}{x}\int_{0}^x h(t)dt=\frac{1}{n}\sum_{k=1}^nh(\lambda_k^* x) +R_n(x),
\end{equation}
where $\lambda_k^*$ are generated by the frequencies in (\ref{Cheb_quadr}) appropriately shifted to a neighbourhood of $(0,1)$. Note that the asymptotic behaviour of $\lambda_k$ and thus $\lambda_k^*$ is fully studied in \cite{Kuzmin}.
\medskip

In a similar manner Theorem~\ref{theorem2} leads to Chebyshev-type quadrature formulas for  integrals (\ref{integral}) with other weights $\rho$.

\medskip

For (\ref{integral}) with $\rho(x)\equiv 1$, one can find quadratures based on the $h$-sums in \cite{Dan2008} but they still require more arithmetic operations than $H_n$  (see Section~\ref{section3.2}). If the integral (\ref{integral}) with $\rho(x)\equiv 1$ is interpolated by (\ref{AFS}), then one obtains the well-known Gauss quadrature (see \cite[Subsection 2.2]{DanChu2016}). If $\rho(x)=(1-x^2)^{-1/2}$ in (\ref{integral}), the corresponding Gauss-type quadrature (usually called Gauss-Chebyshev or Hermite quadrature) has equal amplitudes as $H_n$ does (see \cite[Sunsection 2.2]{DanChu2016} and \cite[Section VI, \S 4]{Natanson}). However, these quadratures have different nature, namely, the ones based on $H_n$ have equal amplitudes for any weight~$\rho$ in the integral (\ref{integral}), whilst the ones based on (\ref{AFS}) have this property only for  $\rho(x)=(1-x^2)^{-1/2}$ as shown by K. Posse and J. Geronimus, see \cite[Section VI, \S\S 4--5]{Natanson}.

\subsubsection{Numerical differentiation in a neighbourhood of $z=0$}
\label{diff_section} Now let us interpolate
$$
f(z)=h_0t+zh'(z)=h_0t+\sum_{m=1}^\infty mh_m z^m,\qquad h_0\neq 0,
$$
where $t>0$ is parameter, by sums $H_n$. We clearly have $r_m=m$ for $m=1,2,\ldots$ and thus
$$
\mu=t>0,\qquad s_m=\frac{n}{t}\, m, \qquad m=1,\ldots,n.
$$
Solving the system (\ref{Newton_MP}) leads to the identity
$$
h_0t+zh'(z)=\frac{t}{n}\sum_{k=1}^nh(\lambda_k z) +R_n(z),
$$
where $\lambda_k=\lambda_k(t,n)$ are independent of $h$ and  are universal in this sense. Finally, the following interpolation formula holds true:
\begin{equation}
\label{derivative}
zh'(z)= t \left(-h(0)+\frac{1}{n}\sum_{k=1}^nh(\lambda_k z)\right) +R_n(z),\qquad R_n(z)=O(z^{n+1}),
\end{equation}
that is exact for polynomials $h$ of degree $\le n$, i.e. $R_n(z)\equiv 0$ in that case. 

Now we estimate the remainder and $|\lambda_k|$ in (\ref{derivative}) using Theorem~\ref{theorem_Pade_remainder}$(b)$ and that $|r_m|=m\le \frac{t}{n} \gamma ma^m$ with $\gamma=n$ and $a=t^{-1/n}$, where $t\ge 1$. 
Fix $n$ and suppose that $|h_m|\le 1$ for all $m$. Then by Theorem~\ref{theorem_Pade_remainder}$(b)$,
$$
\max_{k=1,\ldots,n}|\lambda_k|\le \frac{2n+1}{t^{1/n}}. 
$$
This, in particular, implies that $\max_{k=1,\ldots,n}|\lambda_k|\to 0$ as $t\to \infty$,
i.e. the nodes $\lambda_k=\lambda_k(t)$ in  (\ref{derivative}) tend to $z=0$ as $t$ grows.  A similar behaviour of nodes is observed in \cite{DanChu2011,DanChu2016} in numerical differentiation formulas based on amplitude and frequency sums and $h$-sums. Unfortunately, there is a compensation of this phenomenon: $\mu=t \to \infty$ as $t\to \infty$.

Furthermore,  we deduce for $|z|<t^{1/n}/(2n+1)$ from Theorem~\ref{theorem_Pade_remainder} that
$$
 |R_n(z)|\le \frac{2t|(2n+1)t^{-1/n}z|^{n+1}}{(1-(2n+1)t^{-1/n}|z|)^2}= \frac{2t^{-1/n}|(2n+1)z|^{n+1}}{(1-(2n+1)t^{-1/n}|z|)^2},\qquad n\ge 2.
$$

Say, if $t=2^n$, then it holds for (\ref{derivative}) and $|z|<1/(n+\tfrac{1}{2})$ that
$$
\max_{k=1,\ldots,n}|\lambda_k|\le n+\tfrac{1}{2},\qquad |R_n(z)|\le \frac{|(2n+1)z|^{n+1}}{(1-(n+\tfrac{1}{2})|z|)^2}.
$$

\medskip

Formulas similar to (\ref{derivative}) are obtained in \cite{Dan2008,Chu2010}. Again, they require more arithmetic operations than (\ref{derivative}), although have almost the same interpolation rate, $O(z^{n})$. An analogous problem for amplitude and frequency sums (with the remainder $O(z^{2n})$) is not solvable at all and can be managed only after proper regularisation \cite[Section 5]{DanChu2016}.

\section{Prony interpolation by $H_n$}
\label{Prony_section}

Now we use the results and remarks given above for the most important part of our exposition --- the Prony-type interpolation by usual (i.e. not generalized) exponential sums with equal weights. Recall that interpolation by exponential sums has many practical applications, e.g. in analysis of time series, and is now widely studied  (see \cite{Lyubich,Prony_method_survey_2,Numerical_Prony_method_survey_1,Numerical_Prony_method_survey_2,Numerical_Prony_method_survey_3,Numerical_Prony_method_survey_4,DanChu2016} and references therein).

\subsection{Main theorem about the Prony interpolation by $H_n$}
\label{section4.1}

Recall that we deal with the sums (\ref{equal_AFS}), where $h(z)=\exp(z)$, i.e. with the sums (\ref{exp-sums}): 
$$
H_n^{\sf exp}(z)=\frac{\mu}{n}\sum_{k=1}^n \exp(\lambda_k z).
$$
Within this framework, we aim to interpolate the table (\ref{table}):
$$
\left\{m, g(m)\right\}_{m=0}^n,\qquad g(0)\neq 0,
$$
where the sequence $\{g(m)\}_{m=0}^n$ is generated by a complex-valued function $g$, see~(\ref{Prony-problem}). Before moving forward, recall the original Prony exponential interpolation (\ref{Prony_exponential}) and the information around (\ref{Prony_exponential}).
\begin{theorem}
\label{theorem5}
Given a table $(\ref{table})$, there always exist uniquely determined $($up to a period of the complex exponent$)$ numbers $\mu\neq 0$ and $\Lambda_n=\{\lambda_k\}_{k=1}^n$, with $\lambda_k\in \overline{\mathbb{C}}$, such that
\begin{equation}
\label{exp_prob}
H_n^{\sf exp}(m)=\frac{\mu}{n}\sum_{k=1}^n \exp(\lambda_k m)=g(m),\qquad m=0,\ldots,n.
\end{equation}

More precisely, the numbers can be determined as follows:
$$
\mu=g(0)\qquad \text{and}\qquad \exp(\lambda_k)= l_k,\qquad k=1,\ldots,n,
$$
where $l_k\in \mathbb{C}$, $k=1,\ldots,n$, are the solutions to the Newton-type moment problem
\begin{equation}
\label{eq_l}
\sum_{k=1}^n l_k^m=\frac{n}{g(0)}\,g(m),\qquad m=1,\ldots,n.
\end{equation}

Additionally, 
\begin{itemize}
\item[(a)] if $|g(m)|\le \frac{|g(0)|}{n}\,a^m$ for some $a\ge 0$ and all $m=1,\ldots,n$, then\footnote{One can take into account the periodicity of the exponential function to suppose that $|\Im \lambda_k|\le \pi$.}
$$
\max_{k=1,\ldots,n}|l_k|\le (1+\varepsilon_n)a\quad \Rightarrow \quad -\infty\le \Re \lambda_k \le \ln a +\varepsilon_n;
$$
\item[(b)] if $|g(m)|\le \frac{|g(0)|}{n}\,\gamma ma^m$ for some $\gamma>0$ and $a\ge 0$ for all $m=1,\ldots,n$, then
$$
\max_{k=1,\ldots,n}|l_k|\le (1+2\gamma) a \quad \Rightarrow \quad -\infty\le \Re \lambda_k \le \ln a +\ln(1+2\gamma).
$$
\end{itemize}
\end{theorem}
\begin{proof} From (\ref{exp_prob}) with $m=0$ we immediately get $\mu=g(0)\neq 0$. Then for $m=1,\ldots,n$ in  (\ref{exp_prob}) we use the idea from the original Prony method consisting in the exchange $ \exp(\lambda_k)= l_k$ to obtain the system (\ref{eq_l}). This is actually the system (\ref{Newton_MP}) with $s_m=\frac{n}{g(0)}\,g(m)$ that  always  has a unique (complex) solution $\{l_k\}_{k=1}^n$. We may then determine $\lambda_k$ by assuming
\begin{equation}
\label{lambda_LN}
\lambda_k:=\left\{
\begin{array}{ll}
-\infty, & l_{k}=0,\\
{\rm ln\,} l_k, & l_k\neq 0,\\
\end{array}
\right.
\qquad k=1,\ldots,n.
\end{equation}

The clauses $(a)$ and $(b)$ follows from Theorems~\ref{theorem1} and \ref{theorem_new_estimate} and that $|l_k|=e^{\Re \lambda_k}$.
\end{proof}

Let us emphasize that we are unaware of any results similar to Theorem~\ref{theorem5} although the idea behind it is very close to Prony's one.

\subsection{Comparison with the original Prony  exponential interpolation problem}
\label{section4.2}
Summarising the previous subsection, the Prony interpolation problem (\ref{Prony-problem}) is {\it always solvable in a unique way}. What is more, {\it $|e^{\lambda_k}|$, $|\mu|$ and $|\lambda_k|$ can be efficiently estimated} under several natural assumptions on the sequence $\{g(m)\}$. 

This is in sharp contrast to the original Prony problem (\ref{Prony_exponential}). Recall that by the exchange  $\exp(\lambda_k)= l_k$ one can easily come from (\ref{Prony_exponential}) to the polynomial (\ref{generating_polynomial}) and the system (\ref{Prony_MP}), where $s_m$ should be exchanged for $g(m)$. Consequently, the Prony problem (\ref{Prony_exponential}) can be unsolvable in the general case, as follows from the discussion in Section~\ref{section3.2}. Then the numerical methods, mentioned in Section~\ref{section3.2}, hardly can help as the corresponding iterative processes become divergent if the corresponding error (residual) is required to vanish. Several theoretical examples can be found in \cite[Section 7]{DanChu2016} to confirm this statement. Indeed, there are examples of $s_m$ such that $|s_m|\le 1+\varepsilon$, where $\varepsilon>0$, and $\mu_1=\mu_1(\varepsilon)\to \infty$ or $\lambda_1=\lambda_1(\varepsilon)\to \infty$ as $\varepsilon \to 0$. Some general results on this can be also found in \cite{Numerical_Prony_method_survey_1}.

Furthermore, in spite of the huge bibliography related to the Prony method and generalized exponential sums (see \cite{Lyubich,Prony_method_survey_2,Numerical_Prony_method_survey_1,Numerical_Prony_method_survey_2,Numerical_Prony_method_survey_3,Numerical_Prony_method_survey_4,DanChu2016} and references therein), we could not find any more or less general estimates for amplitudes and frequencies similar to those in Theorem~\ref{theorem5}. Probably, they just do not exist because of the above-mentioned divergence examples from \cite[Section 7]{DanChu2016} and the results from \cite{Numerical_Prony_method_survey_1}. As for particular cases, several estimates were obtained in \cite[Sections 5 and 6]{DanChu2016} for special sequences. Moreover, some conclusions about $\mu_k$ and $\lambda_k$ (e.g. that they are real, positive or belonging to the segment $[0,1]$) can be made if $g(m)$ satisfies the criteria due to Hamburger, Stieltjes or Hausdorff, related to the classical moment problems \cite{Classical_MP}, see also \cite[Chapter VI, \S3]{Braess}. Furthermore, some nice estimates can be directly derived from properties of the roots of some classical orthogonal polynomials of the form (\ref{generating_polynomial}) generated by properly chosen sequences $\{s_m\}$, see e.g. \cite{Lyubich} for the connection of (\ref{AFS}) and classical orthogonal polynomials.

\subsection{Examples and further remarks}
\label{section4.3}
 We start with several simple examples.
\begin{example}[Chebyshev's quadrature nodes]
	Let us interpolate the table
	$$
	\left\{m,\frac{1+(-1)^m}{m+1}\right\}_{m=0}^n
	$$
	by exponential sums $H_n^{\sf exp}$. By Theorem~\ref{theorem5}, we get $\mu=2$ and thus need to solve the system
	$$
	\sum_{k=1}^n l_k^{m}=\frac{n}{2}\frac{1+(-1)^m}{m+1},\qquad m=1,\ldots,n.
	$$
We have already considered it above, see (\ref{Cheb_syst}). Indeed, $\l_k$ are then the nodes in Chebyshev's quadrature (\ref{Cheb_quadr}). Then by (\ref{lambda_LN}) we obtain $\Lambda_n$.	
\end{example}
\begin{example}
If $g(z)=1/(z+1)$, then the table to interpolate is
$$
\left\{m,1/(m+1)\right\}_{m=0}^n.
$$
Clearly, $\mu=1$ and $s_m=\frac{n}{m+1}$, $m=1,\ldots,n$.
We considered this $\{s_m\}$ already around (\ref{Cheb_syst_01}) and mentioned that the corresponding solution to (\ref{Newton_MP}) is produced by the nodes of shifted Chebyshev's quadrature (\ref{Cheb_quadr_shifted}). Moreover, the behaviour of the nodes was completely studied in \cite{Kuzmin}. In particular\footnote{Very roughly speaking, $\Re l_k\in (-3\sqrt{\ln n}/n,1+3\sqrt{\ln n}/n)$ and $\Im l_k \in (-\tfrac{1}{4},\tfrac{1}{4})$.}, one can deduce from \cite[\S 7]{Kuzmin} that for $n\ge n_0$,
$$
\max_{k=1,\ldots,n} |l_k|\le 1+\frac{3\ln n}{n} \quad \Rightarrow \quad -\infty\le \Re \lambda_k \le \frac{3\ln n}{n}.
$$

Consequently, with these $\lambda_k$,
$$
\frac{1}{z+1}=H_n^{\sf exp}(z)=\frac{1}{n}\sum_{k=1}^n  \exp(\lambda_k z),\qquad z=0,1,\ldots,n.
$$

Thus we constructed exponential sums for the function $g(z)=1/(z+1)$ with $z\ge 0$. This problem, especially for exponential sums $\mathcal{H}_n(z;\exp)$, attracts much attention of different authors, see e.g. \cite{Beylkin_1/x,Hackbusch} and references therein. It is an independent interesting question to compare the above-mentioned interpolants $H_n^{\sf exp}$ with the ones based on other exponential sums.
\end{example}
\begin{example}
If $g(z)=c$, where $c\neq 0$ is a constant, then
$$
\mu=c,\qquad s_m=n,\qquad m=1,\ldots,n.
$$
Clearly, then $l_k=1$ for $k=1,\ldots,n$, and thus
$\Lambda_n=\{0\}_{k=1}^n$ and
$H_n^{\sf exp}(z)\equiv c$.

Note that the original Prony problem (\ref{Prony_exponential}) is not solvable in this case under the assumption that $\mu_k\neq 0$ and $\lambda_k$ are pairwise distinct. If the assumption is relaxed though, one gets the same result.
\end{example}
\begin{example}
Let $g(z)=z+1$. Thus the table to interpolate is
$$
\left\{m,m+1\right\}_{m=0}^n.
$$
Clearly, $\mu=1$ and $s_m=n(m+1)$ for $m=1,\ldots,n$. From this we can find $\Lambda_n$ to construct the required $H_n^{\sf exp}$. 
From Theorem~\ref{theorem5}$(b)$ for $\gamma=2n$ and $a=1$ we get the estimates
$$
\max_{k=1,\ldots,n}|l_k|\le 1+4n \quad \Rightarrow \quad -\infty\le \Re \lambda_k \le \ln(1+4n).
$$
These estimates however are quite pessimistic as computer experiments suggest. For instance, for $n\le 50$ calculations show that
$\max_{k=1,\ldots,n}|l_k|< 9/2$ and $\Re \lambda_k \in (0,3/2)$. What is more, $l_k$ seem to be settled on a kind of cardioid with a cusp at the origin as $n\to \infty$.

Note that the Prony problem (\ref{Prony_exponential}) is not solvable for the table under consideration. 
\end{example}

To finish the discussion, we make several remarks.

\begin{remark}
For $h$-sums of the form (\ref{h-sums-nu}) with $\eta=1,2,\ldots$ and $z=m$ we get
$$
\mathcal{H}^*_{\eta,n}(m;\exp)=\sum_{k=1}^n \lambda_k^\eta \exp(\lambda_k m),\qquad m=0,1,\ldots.
$$
Unfortunately, in this case the exchange  $\exp(\lambda_k)= l_k$ does not lead to any familiar system of equations and the corresponding interpolation problem remains unsolved. This is another advantage of $H_n$ over $\mathcal{H}^*_{\eta,n}$ within the Prony problem context.
\end{remark}
\begin{remark}
In the case of the table $\{x_m,g(m)\}_{m=0}^n$ for $n+1$ equidistant nodes $x_m:=a+(b-a)\tfrac{m}{n}\in [a,b]$, $m=0,\ldots,n$,  one should consider the sums 
$$
H_n^{\sf exp}(z;[a,b]):=\frac{\mu}{n}\sum_{k=1}^n \exp\left(\lambda_k \frac{ n(z-a)}{b-a}\right)
$$
instead of $H_n^{\sf exp}$. Indeed, 
$$
H_n^{\sf exp}(x_m;[a,b])=\frac{\mu}{n}\sum_{k=1}^n \exp (\lambda_k m)
=g(m),\qquad m=0,\ldots,n,
$$
and one can proceed as in Theorem~\ref{theorem5}.
\end{remark}

\begin{remark}
Since the Newton moment problem (\ref{Newton_MP}) always has a unique solution, in contrast to the system (\ref{syst_AFSs}), one can possibly use/adapt the numerical methods for (\ref{syst_AFSs}) (e.g. ESPRIT or MUSIC, see \cite{Numerical_Prony_method_survey_1,Numerical_Prony_method_survey_2,Numerical_Prony_method_survey_3,Numerical_Prony_method_survey_4}) for solving (\ref{Newton_MP}).  Recall that then there is no divergence problem as for  unsolvable systems (\ref{syst_AFSs}).

Furthermore, as in the case of overdetermined systems (\ref{syst_AFSs}), i.e. with $M>2n$ equations instead of $2n$, one can use numerical methods (see \cite{Numerical_Prony_method_survey_1,Numerical_Prony_method_survey_2,Numerical_Prony_method_survey_3,Numerical_Prony_method_survey_4}) to find approximate solutions to overdetermined systems (\ref{syst_h_sums}) with $M>n$ equations. This would allow to approximately solve an overdetermined interpolation problem of type (\ref{table}) for the sums (\ref{equal_AFS}).

The above-mentioned are interesting practical questions that are however out of scope of the current paper as we deal only with analytical methods here.
\end{remark}

\begin{remark} It is recently shown in \cite{Komarov2018} that for any sequence $\{\tilde{g}(m)\}_{m=1}^n$ and sufficiently large $n$ there exist pairwise distinct numbers $l_k$, $k=1,\ldots,n$, such that
$$
\sum_{k=1}^{2n+1}l_k^m=\tilde{g}(m),\qquad m=1,\ldots,n,\quad \text{and}\quad |l_k|=1,\quad k=1,\ldots,2n+1.
$$
This implies in the context of our exponential interpolation that there are $\{\phi_k\}_{k=1}^{2n+1}$ such that $\phi_k\in [0,2\pi)$ and any table  $\{m,g(m)\}_{m=0}^n$ with $g(0)\neq 0$ can be interpolated by
$$
\frac{\mu}{2n+1}\sum_{k=1}^{2n+1}\exp (\phi_k i z)=\frac{\mu}{2n+1}\sum_{k=1}^{2n+1}\left(\cos(\phi_kz)+i\sin(\phi_kz)\right).
$$
\end{remark}

\section{The proof of Theorem~\ref{theorem1}}
\label{section_proof_of_theorem1}

We first recall  the following result.
\begin{lemma}[see \cite{Chu2010}]
If $|S_m|\le a^m$ for some $a\ge 0$ and all $m=1,\ldots,n$, then
\begin{equation}
\max_{k=1,\ldots,n}|\lambda_k|\le (1+\varepsilon_n)a,
\end{equation}
where $\varepsilon_n\in (0,1)$ and satisfies the equation
\begin{equation}
\label{equal_eps}
\varepsilon_n^2-(1-\varepsilon_n)^{n+1}=0.
\end{equation}
\end{lemma}

There exist several estimates for $\varepsilon_n$ in (\ref{equal_eps}). In particular, it was shown in  \cite{Chu2010} that $\varepsilon_n=o(n^{-\beta})$, $n\to\infty$, for any fixed $\beta \in (0,1)$. Later on, it was proved in~\cite{DanChu2011} that
$\varepsilon_n< n^{-1}\ln^2 n$, $n\ge 10$. Further estimates were announced (with some gaps in the proof though) in the manuscript \cite{Chu2013}.

Our purpose now is to obtain final estimates for $\varepsilon_n$. We start with the following lemma that contains the first part of Theorem~\ref{theorem1}.

\begin{lemma}
	\label{lemma3.3}
It holds for $\varepsilon_n$ in $(\ref{equal_eps})$ that
	\begin{equation}
	\label{eps_n_ineq}
	\varepsilon_n\le\frac{2(\ln n -\ln\ln n)}{n}< \frac{2\ln n}{n}, \quad n\ge 2, \qquad  \varepsilon_n\sim \frac{2\ln n}{n}, \quad n\to\infty.
	\end{equation}
\end{lemma}
\begin{proof}
Let us prove the inequality in (\ref{eps_n_ineq}) for $n\ge 2$. For this, consider the function
$$
E(x):=x^2-(1-x)^{n+1},\qquad x\in [0,1].
$$
Since $E'(x)=2x+(n+1)(1-x)^n>0$ for $x\in[0,1]$,
the function $E$ monotonically increases in the segment $[0,1]$. Moreover, $E$ has different signs at the ends of the segment. Consequently, in order to obtain the required estimate, it is sufficient to
prove the inequality
$$
E \left(\frac{2\ln (n/\ln n)}{n}\right)> 0.
$$
Take into account that $1-x\le e^{-x}$ and $(1-x)^{n+1}\le (1-x)^n$ for $x\in [0,1]$. Thus
$$
E\left(\frac{2\ln (n/\ln n)}{n}\right)\ge \left(\frac{2\ln (n/\ln n)}{n}\right)^2-e^{-n\cdot \frac{2\ln (n/\ln n)}{n}}=
\left(\frac{2\ln (n/\ln n)}{n}\right)^2-\left(\frac{\ln n}{n}\right)^2
$$
$$
=\frac{\left(\ln \left(\tfrac{n}{\ln n}\right)^2+\ln n\right)\cdot \left(\ln \left(\tfrac{n}{\ln n}\right)^2-\ln n\right)}{n^2}=
\frac{\ln \tfrac{n^3}{(\ln n)^2}\cdot \ln \tfrac{n}{(\ln n)^2}}{n^2}>0,\qquad n\ge 2.
$$

To prove $\varepsilon_n\sim \frac{2\ln n}{n}$, $n\to\infty$, we approximately solve the equation (\ref{equal_eps}) with respect to $\varepsilon_n$. Let $\varepsilon_n=C_n\frac{\ln n}{n}$. From the inequality in (\ref{eps_n_ineq}) that we just proved it follows that $0<C_n<2$ for $n\ge 2$. Substituting the expression for $\varepsilon_n$ into (\ref{equal_eps}) and taking the logarithm of the equality obtained leads to
$$
2\left(\ln C_n +\ln \ln n-\ln n \right)=(n+1)\ln \left(1-\frac{C_n \ln n}{n}\right).
$$
Therefore for $n\to \infty$,
$$
O(1) +2\ln \ln n-2\ln n =(n+1)\left(-\frac{C_n \ln n}{n}+o\left(\frac{\ln n}{n}\right)\right).
$$
Dividing both parts by $\ln n$ implies after several simplifications that $C_n=2-o(1)$ and 
$$
\varepsilon_n=\frac{2\ln n}{n} -o\left(\frac{\ln n}{n}\right), \qquad n\to\infty.
$$
Thus we are done.
\end{proof}

The second part of Theorem~\ref{theorem1} is covered by the following result that was first announced in the manuscript \cite{Chu2013}.
\begin{lemma}
For odd $n\ge n_0$ there exists $\Lambda_n$ such that $|S_m(\Lambda_n)|\le a^m$ and
\begin{equation}
\label{prim1.2}
\left(1+\frac{1}{10n}\right)a\le |\lambda_1|\le  \left(1+\frac{1}{n}\right)a.
\end{equation}
\end{lemma}
\begin{proof}
By changing variables we come to the case $a=1$.
For $n\ge 2$ consider the polynomial (\ref{P_n}) with the roots $\Lambda_n$ whose power sums are defined by
$$
{S}_m=1, \quad m=1,\ldots,n-1, \qquad {S}_n=(-1)^n.
$$
It can be easily seen that for even $n$ (we do not consider this case below) one has $$
p_n(\lambda):=\lambda^{n}-\lambda^{n-1},
$$
whose roots lie in the disc $|\lambda|\le 1$. For odd $n$,
$$
S_m=1,\quad m=1,\ldots,n-1,\qquad S_n=-1,
$$
and, by (\ref{sigma}),
$$
\sigma_1=1,\quad \sigma_2=\ldots={\sigma}_{n-1}=0,\quad \sigma_n=-2/n.
$$
Consequently,
$$
{P}_n(\lambda):=\lambda^{n-1}(\lambda-1)+2/n.
$$
Let us show that one of the roots of this polynomial, say, ${\lambda}_1$, satisfies~(\ref{prim1.2}). Below we use the notation
$$
D(r):=\{\lambda:|\lambda|<r\}, \quad \gamma(r):=\partial D(r), \quad
l(r):=p_n(\gamma(r)), \quad L(r):=P_n(\gamma(r)).
$$
		
We first prove the right hand side inequality in (\ref{prim1.2}). By Rouch\'e's theorem, for $n\ge 5$ the polynomials ${P}_n$ and $p_n$ have the same number of roots in the disc $D\left(1+ 1/n\right)$. Indeed, for  $\lambda\in \gamma\left(1+1/n\right)$ we have
$$
|{P}_n(\lambda)-p_n(\lambda)|=\tfrac{2}{n}<\left(1+\tfrac{1}{n}\right)^{n-1}\tfrac{1}{n}\le |p_n(\lambda)|,\qquad n\ge 5.
$$
Consequently, for odd $n$ the roots of ${P}_n$ satisfy the estimate
$$
\max_{k=1,\ldots,n}|{\lambda}_k|< 1+\tfrac{1}{n}, \qquad n\ge 5.
$$

It is clear geometrically that the argument of the vector  $w_1=p_n(\lambda)$ is monotonically growing  while moving around the circle $\gamma(r)$ with $r>1$ in the positive direction. Moreover, the length of the vector $w_1$ is growing while $\lambda$ moves around the upper semicircle  $\gamma(r)\cap \mathbb{C}^+$ in the positive direction. The image $l(r)$ of the circle $\gamma(r)$ is symmetric with respect to the real axis and has $n$ self-intersection points, belonging to the axes. These points divide the curve $l(r)$ into $n$ connected components (loops), each containing the origin. Note also that the image of $L(r)$ is the curve $l(r)$ shifted  to the right by $2/n$. Consequently, the corresponding loops of  the image  $L\left(1+1/n\right)$ still contain the origin.
		
Now we are going to show that at least one of the loops of the image $L\left(1+1/(10n)\right)$ of the circle $\gamma\left(1+1/(10n)\right)$ does not contain the origin. This means that the argument increment of  the   vector $w_2={P}_n(\lambda)$ does not exceed  $2\pi(n-1)$ on the circle $\gamma\left(1+1/(10n)\right)$, and thus at least one of the roots of the polynomial  ${P}_n$ lie outside the circle, i.e. the left hand side estimate in (\ref{prim1.2}) is true. Consider the arc
		$$
		\gamma^*:=\left\{\lambda \in \gamma\left(1+\tfrac{1}{10n}\right): -\tfrac{17}{10n}\le\arg \lambda \le \tfrac{17}{10n}\right\}.
		$$
Let us find the argument increment over this arc for the continuous branch of the argument of $w_1$. It is equal to the sum of the argument increments for each factor in $p_n$, i.e.
		$$
		\Delta_{\gamma^*} \arg w_1=\Delta_{\gamma^*} \arg  \lambda^{n-1}+\Delta_{\gamma^*} \arg  (\lambda-1).
		$$
It can be easily seen that  $\Delta_{\gamma^*} \arg  \lambda^{n-1}=\frac{17}{5}(1-1/n)>1.08\pi (1-1/n)$. Moreover, the increment $\Delta_{\gamma^*} \arg  (\lambda-1)\ge 2\,{\rm arctan}\, 16>0.96\pi$ for sufficiently large $n$. This follows from the fact that,  for $\lambda\in \gamma^*$ and $\varphi=\pm \frac{17}{10n}$,  we have  $0<{\rm Re\,}(\lambda-1)\le 1/(10n)$ and $|{\rm Im\,}(\lambda-1)|\le 18/(10n)$, where ${\rm Im\,}(\lambda-1)$ has the same sign as $\varphi$. Thus the total increment  of the argument of $\Delta_{\gamma^*} \arg w_1>2\pi$ for sufficiently large  $n$. This implies that the image $p_n(\gamma^*)$ includes a loop that contains the origin inside.

Now let us prove that that the analogous loop of the image $P_n(\gamma^*)$ already does not contain the origin inside. To do so, let us note that the image $p_n(\gamma^*)$ entirely lies in the disc $|\lambda|<2/n$. Indeed, for $\lambda\in \gamma^*$ and sufficiently large $n$ we have
		\begin{align*}
		&|\lambda|^{n-1}\le \left(1+\tfrac{1}{10n}\right)^{n-1}< \sqrt[10]{e},\\
		&|\lambda-1|\le \sqrt{\left(1+\tfrac{1}{10n}\right)^2-2\left(1+\tfrac{1}{10n}\right)\cos \tfrac{17}{10n} +1}\\
		&\qquad\quad \le \sqrt{\tfrac{1}{(10n)^2}+\left(\tfrac{17}{10n}\right)^2\left(1+\tfrac{1}{10n}\right)}<
		\tfrac{18}{10n}.
		\end{align*}
Consequently,  $|p_n(\gamma^*)|<18\sqrt[10]{e}/(10n)<2/n$ for sufficiently large  $n$ and therefore the image $P_n(\gamma^*)$ entirely lies in the disc $|\lambda-2/n|<2/n$ that does not contain the origin.
		
Summarising, for odd $n\ge n_0$ the power sums of the roots of $P_n$ satisfy the inequalities $|{S}_m|\le 1$ for $m=1,\ldots,n$, and one of the roots meets the estimate (\ref{prim1.2}). 
\end{proof}

\end{document}